\documentclass[11pt,a4paper]{amsart}
\usepackage[active]{srcltx}
\usepackage{amsmath,amsfonts,amssymb,epsf,mathrsfs,amsthm,lmodern}
\usepackage[normalem]{ulem}
\usepackage[T1]{fontenc}
\usepackage[utf8x]{inputenc} 
\usepackage[backref=page,bookmarksopen=true]{hyperref}

\usepackage{enumitem}

\let\bbbibitem\bibitem
\renewcommand{\bibitem}[2][]{\bbbibitem[#1]{#2}\label{#2}} 

\def\fin{\hfill\hbox{\hskip .2cm $\square$}\medskip}

\theoremstyle{plain}
\newtheorem{theo}{Theorem}[section]

\newtheorem{coro}[theo]{Corollary}

\theoremstyle{definition}

\newtheorem{fact}[theo]{Fact}

\theoremstyle{remark}
\newtheorem{rema}[theo]{Remark}

\def\cal{\mathcal}

\def\g{\gamma}
\def\G{\Gamma}

\def\Ad{{\rm Ad}}

\def\PU{{\rm PU}}
\def\SU{{\rm SU}}

\def\SO{{\rm SO}}

\def\C{{\mathbb C}}

\def\N{{\mathbb N}}
\def\Z{{\mathbb Z}}
\def\R{{\mathbb R}}
\def\fd{\longrightarrow}
\def\pfd{\rightarrow}

\def\la{\langle}
\def\ra{\rangle}

\def\X{{\cal X}}
\def\Y{{\cal Y}}

\def\gg{{\mathfrak g}}
\def\kg{{\mathfrak k}}
\def\mg{{\mathfrak m}}
\def\pg{{\mathfrak p}}
\def\hg{{\mathfrak h}}
\def\qg{{\mathfrak q}}
\def\ug{{\mathfrak u}}

\def\ag{{\mathfrak a}}
\def\rg{{\mathfrak r}}
\def\sg{{\mathfrak s}}

\def\om{\omega}

\def\comp{{C}}
\def\stabY{{K_{\cal Y}}}
\def\mm{\Omega}
\def\D{{\cal D}}

\def\Ho{{{\mathbb H}^{1}_{\mathbb C}}}

\def\u{{\mathfrak u}}

\def\lg{{\mathfrak l}}

\def\Y{{\mathcal Y}}

\begin{document}

\title[On the equidistribution of totally geodesic submanifolds]{On the equidistribution of
  totally geodesic submanifolds in locally symmetric spaces and application to boundedness
  results for negative curves and exceptional divisors}

\author{Vincent Koziarz} 
\author{Julien Maubon}
\address[Vincent Koziarz]{Univ. Bordeaux, IMB, CNRS, UMR 5251, F-33400 Talence, France}
\email{vkoziarz@math.u-bordeaux1.fr}
\address[Julien Maubon]{Institut \'Elie Cartan, UMR 7502,  Universit\'e
   de Lorraine, B. P. 70239, F-54506 Vand\oe uvre-l\`es-Nancy Cedex,
    France}
  \email{julien.maubon@univ-lorraine.fr}


\sloppy

\begin{abstract}      
We prove that on a smooth
complex surface which is a
compact quotient of the bidisc or of the 2-ball, there is at most a
finite number of totally geodesic curves with negative self intersection.  More generally, there are
only finitely many exceptional totally geodesic divisors in a compact Hermitian 
locally symmetric space of the noncompact type of dimension at least 2.  
This is deduced from a convergence result for currents of integration along totally geodesic
subvarieties in compact Hermitian locally symmetric spaces which itself follows from  
 an equidistribution theorem for totally geodesic submanifolds in a locally symmetric
space of finite volume.
\end{abstract}

\maketitle

\section{Introduction}

Our motivation for writing this note comes from a question about totally geodesic curves in compact
quotients of the $2$-ball related to the so-called Bounded Negativity Conjecture. This conjecture
states that if $X$ is a smooth complex projective surface, there exists a number $b(X)\geq 0$ such
that any negative curve on $X$ has self-intersection at least $-b(X)$. On a Shimura surface $X$, i.e.\
an arithmetic compact quotient of the bidisc or of the $2$-ball, one can
ask wether such a conjecture holds for Shimura (totally geodesic) curves. In~\cite{Bal}, using an
inequality of Miyaoka~\cite{Miyaoka}, it was proved that on a 
quaternionic Hilbert modular surface, that is, a compact 
quotient of the bidisc, there are only a finite number of negative Shimura curves. The same question
for Picard modular surfaces, i.e.\ quotients of the $2$-ball, was open as we learned from
discussions with participants of the MFO mini-workshops ``K\"ahler Groups''
(\url{http://www.mfo.de/occasion/1409a/www_view}) and ``Negative Curves on Algebraic Surfaces''
(\url{http://www.mfo.de/occasion/1409b/www_view}). See the report~\cite{DKMS} and~\cite[Remarks 3.3
\& 3.7]{Bal}. There was a general feeling that this should follow from an
equidistribution result about totally geodesic submanifolds in locally symmetric
manifolds. Using such a result, we prove that this is indeed true (we include the already known 
case of the bidisc since the same method also implies it):

\begin{theo}\label{negativecurves}
Let $X$ be closed complex surface whose universal cover is biholomorphic to either the 2-ball or
  the bidisc. Then $X$ only supports a finite number of totally geodesic curves with negative self
intersection.  

More generally, let $X$ be a closed Hermitian locally symmetric space of the noncompact type of complex dimension
$n\geq 2$. Then $X$ only supports a finite number of exceptional totally geodesic divisors.  
\end{theo}           

It is known that the irreducible Hermitian symmetric spaces of the noncompact type
admitting totally geodesic divisors are those associated with the Lie groups $\SU(n,1)$, $n\geq 1$, and
$\SO_0(p,2)$, $p\geq 3$, and then that the divisors are associated with the subgroups
$\SU(n-1,1)$ and $\SO_0(p-1,2)$
respectively, see~\cite{Onishchik,BerndtOlmos}. Note however that Theorem~\ref{negativecurves} also 
applies in the 
case of reducible symmetric spaces.

The first assertion of this result has been obtained independently by  
M.~Möller and D.~Toledo~\cite{MT}, who also participated in the aforementioned workshops.  We
refer to their paper for 
background on Shimura surfaces and 
Shimura curves, and in particular for a discussion of the arithmetic quotients of the 2-ball and of
the bidisc  which admit infinite families of pairwise 
distinct totally geodesic curves. Their proof is based on an equidistribution theorem
for curves in 2-dimensional Hermitian locally symmetric spaces~\cite[\textsection 2]{MT}. Here we
have chosen to present a more general result, see Theorem~\ref{ratner} below,  in the hope that it
can be useful in a wider setting (and indeed it implies the second assertion of the theorem).   

\medskip

Henceforth we will be interested in closed totally geodesic (possibly singular) submanifolds in non
positively curved locally symmetric manifolds of finite volume. 

\smallskip

Let $\X$ be a symmetric space of the noncompact type, $G={\rm Isom}_0(\X)$ the connected
component of the isometry 
group of $\X$, 
$\G$ a torsion-free lattice of $G$
and $X$ the quotient locally 
symmetric manifold $\G\backslash\X$. 

Complete connected totally geodesic (smooth) submanifolds of $\X$ are naturally symmetric
spaces themselves and we will call such a subset $\Y$ a symmetric subspace of the 
noncompact type of
$\X$ if as a symmetric space it is of the noncompact type, i.e. it has no Euclidean factor.
Up to the action of $G$, there is only a finite number of symmetric subspaces of the non compact
type in $\X$, see Fact~\ref{finite}. The orbit of $\Y$ under $G$ will be called the {\em kind} of
$\Y$.

A subset $Y$ of $X$ will be called a {\em closed totally geodesic submanifold of the
  noncompact type} of $X$ if it is of the form $\G\backslash\G \cal Y$, where $\cal Y$ is a
symmetric subspace of the noncompact type of $\X$ such that 
if $S_{\Y}< G$ is the stabilizer of $\cal Y$ in
$G$, $\G\cap S_{\Y}$ is a lattice in $S_{\Y}$. The {\em kind} of $Y=\G\backslash\G \cal Y$ is by
definition the kind of $\Y$. 

It will simplify the exposition to consider only symmetric
subspaces of $\X$ passing through a fixed point $o\in\X$. Therefore we define equivalently a 
closed totally geodesic submanifold of the noncompact type $Y$ of $X$ to be a subset of the form
$\G\backslash\G g\cal Y$, where $\cal Y\subset\X$ is a 
symmetric subspace of the noncompact type passing through $o\in\X$, and $g\in G$ is such that 
if $S_{\Y}< G$ is the stabilizer of $\cal Y$ in 
$G$, $\G\cap g S_{\Y} g^{-1}$ is a lattice in $g S_{\Y} g^{-1}$.

Such a  $Y$ is indeed a closed totally geodesic 
submanifold of $X$, which might 
be singular, and it supports a natural probability measure $\mu_Y$ which can be defined as
follows. 
By assumption, the (right) $S_\Y$-orbit $\G\backslash \G g S_\Y\subset\G\backslash G$ is closed
and supports a unique $S_\Y$-invariant probability measure (\cite[Chap. 1]{Rag}).  
We will denote by $\mu_Y$ the probability measure on $X$  whose support is $Y$ and which is defined
as the push forward of the previous measure by the projection $\pi:\G\backslash G\fd
X=\Gamma\backslash G/K$, where $K$ is the isotropy subgroup of $G$ at $o$. In the
special case when $S_\Y=G$, we 
obtain the natural probability measure $\mu_X$ on $X$. 

\medskip

We will say that a closed totally geodesic submanifold of the noncompact type $Y=\G\backslash\G g\cal Y$
as above is a {\em local factor} if $\cal Y\subset\cal X$ is a {\em factor}, meaning that
there exists a totally geodesic isometric embedding $f:\cal Y\times \R\fd\cal X$ such that
$f(y,0)=y$ for all $y\in\Y$.

\medskip

We may now state the following

\begin{theo}\label{ratner}
Let $\X$ be a symmetric space of the noncompact type, $\G$ a torsion-free lattice of the
connected component $G$ of its isometry group and $X$ the quotient manifold $\G\backslash\X$. Let
$(Y_j)_{j\in\N}$ be a sequence of closed totally 
geodesic submanifolds of the noncompact type of $X$. Assume that no subsequence of $(Y_j)_{j\in\N}$ is
either composed of local factors or contained in a closed totally geodesic proper submanifold of
 $X$. 

Then the sequence of probability measures $(\mu_{Y_j})_{j\in\N}$ converges to the probability 
measure $\mu_X$. 
\end{theo}

\begin{rema}\label{original}
  Although we have not been able to find its exact statement in the literature, this theorem is
  certainly known to experts in homogeneous dynamics and follows from several equidistribution
  results originating in the work of M. Ratner on unipotent flows. In the case of special
  subvarieties of Shimura varieties,
  a very similar result has been obtained by L.~Clozel and E.~Ullmo \cite{CU,U}. From the 
  perspective of geodesic flows (which is in a sense orthogonal to unipotent flows),
  Theorem~\ref{ratner} can probably also be deduced from
  A. Zeghib's article~\cite{Zeghib} (at least for ball quotients it can be).  

The proof we give here is based on a result of Y.~Benoist and J.-F.~Quint~\cite{BQ3}, 
see~\ref{BenoistQuint}.         
\end{rema}

\begin{rema}\label{remcent}
A symmetric subspace $\Y\subset\X$ of the noncompact type is the orbit of a point in $\X$ under a
connected semisimple 
subgroup without compact factor $H_\Y$ of $G={\rm Isom}_0(\X)$. The assumption that $\Y$ is not a factor  
means that the centralizer $Z_G(H_\Y)$ of $H_\Y$ in $G$ is compact. 
Another equivalent formulation is that $\cal Y$ is the only totally geodesic 
orbit of $H_\Y$ in $\cal X$. See Fact~\ref{cent} for a proof. 

This assumption seems quite strong but the conclusion of Theorem~\ref{ratner} is false in general 
without it as the following simple example shows. Let
$X=\Sigma_1\times\Sigma_2$ be the product of two Riemann surfaces of genus at least 2 and let
$(z_j)_{j\in\N}$ be a sequence of distinct points in $\Sigma_1$ such that no subsequence is
contained in a proper geodesic of $\Sigma_1$. Set $Y_j=\{z_j\}\times\Sigma_2$. Then for any
subsequence of $(z_j)$ 
converging to some $z\in\Sigma_1$, the corresponding subsequence of measures $\mu_{Y_j}$ converges
to $\mu_{\{z\}\times\Sigma_2}$. 

We observe that for rank 1 symmetric spaces, and in the case of uniform {\em irreducible} lattices of the
bidisc, the assumption is 
automatically satisfied (see the proof of Theorem~\ref{negativecurves} in~\ref{ProofNC}). 

It would be interesting to know whether it is still needed if one assumes e.g. that $\X$ or $\G$ is
irreducible.
\end{rema}

\medskip

In the
case of Hermitian locally symmetric spaces, 
Theorem~\ref{ratner} gives a convergence result for currents of integration along closed complex totally
geodesic subvarieties (suitably renormalized) from which Theorem~\ref{negativecurves} will follow.
Recall that a on a complex manifold $X$ of dimension $n$, a current $T$ of bidegree
  $(n-p,n-p)$ is said to be {\it (weakly) positive} if for any choice of smooth $(1,0)$-forms
  $\alpha_1,\dots,\alpha_{p}$ on $X$, the distribution $T\wedge
  i\alpha_1\wedge\bar\alpha_1\wedge\dots\wedge i\alpha_{p}\wedge\bar\alpha_{p}$ is a positive
  measure.

\begin{coro}\label{ratnercurrents}
Let $X$ and  $(Y_j)_{j\in \N}$ satisfy the assumptions of 
Theorem~\ref{ratner}  and assume in addition that $X$ is a compact Hermitian locally symmetric space
of complex dimension $n$ and that the $Y_j$'s are complex $p$-dimensional subvarieties of $X$ of the
same kind. 

Then there exists a closed positive $(n-p,n-p)$-form $\mm$ on $X$ (in the sense of currents),
induced by a $G$-invariant $(n-p,n-p)$-form on $\X=G/K$, such that
for any $(p,p)$-form $\eta$ on $X$,
$$\lim_{j\rightarrow +\infty} \frac 1{{\rm vol}(Y_j)}\int_{Y_j} \eta=
\frac {1}{{\rm vol}(X)}\int_X \eta\wedge\mm
$$
Moreover, up to a
positive constant, $\mm$ depends only on the kind of the $Y_j$'s and 
if the $Y_j$'s are divisors, i.e. if $p=n-1$, then for any $j$, the $(1,1)$-form $\mm$ restricted to
$Y_j$ does not vanish.
\end{coro}

Since our initial interest was in $2$-ball quotients, we underline that ball quotients $X$ satisfying
the assumptions of this corollary exist: the arithmetic 
manifolds whose fundamental groups are the so-called uniform 
lattices of type I in the automorphism group $\PU(n,1)$ of the $n$-ball     
are examples of manifolds supporting infinitely many complex totally geodesic subvarieties of dimension $k$
for each $1\leq p<n$, not all contained in a proper totally geodesic
subvariety. Moreover any complex $p$-dimensional totally geodesic subvariety of $X$ is itself a
quotient of the $p$-ball and, as already mentioned, is not a local factor in $X$ (because the
$n$-ball is a rank 1 symmetric space).  

In the case of $2$-ball quotients, the form $\mm$ of Corollary~\ref{ratnercurrents} is proportional
to the Kähler form induced by the unique (up to a positive constant) $\SU(2,1)$-invariant
Kähler form on the ball ${\mathbb H}^2_\C$. In the case of quotients of the bidisc, and if $\om$
denotes the unique (up to a positive constant) 
$\SU(1,1)$-invariant Kähler form on ${\mathbb H}^1_\C$, $\mm$ is proportional to the Kähler form induced by
the $\SU(1,1)\times\SU(1,1)$-invariant form $\om_1+\om_2$ on the bidisc ${\mathbb
  H}^1_\C\times {\mathbb H}^1_\C$, where $\om_1$,
resp. $\om_2$, means $\om$ on the first, resp. second, factor. See Section~\ref{proofRC}.

\medskip

\noindent{\em Acknowledgments.} We are very indebted to Jean-Fran\c cois Quint for explaining to us
his work with Yves Benoist and for numerous conversations on related subjects. We also thank
Ngaiming Mok who pointed out that our main theorem implies the second assertion of
Theorem~\ref{negativecurves}  
and Bruno Duchesne for helpful discussions.

\section{Preliminary results}

For the convenience of the reader, we prove here some more or less well-known and/or easy facts that
will be used in the rest of the paper. 

\medskip

As in the introduction $\X$ is a symmetric space of the noncompact type,  $G$ is the connected
component of the isometry group of $\X$ (it is therefore a semisimple 
real Lie group without compact factor and the connected component of the real points of a
  semisimple algebraic 
group defined over $\R$), $o\in\X$ is a 
fixed origin, $K< G$ is 
the isotropy group of $G$ at $o$, so that  $K$ is a maximal compact subgroup of $G$ and $\cal X=G/K$.

We write $\gg=\kg\oplus\pg$ for the Cartan
decomposition of the Lie algebra $\gg$  of $G$ given by the geodesic symmetry $s_o$ around
$o\in\X$.

We let $\Y\subset \X$ be a symmetric subspace of the noncompact type containing the point $o$. Its
tangent space at $o$ can be identified with a Lie triple system $\qg\subset\pg$, so that setting
$\lg:=[\qg,\qg]\subset\kg$, $\hg:=\lg\oplus\qg$ is a semisimple Lie subalgebra of $\gg$. The
corresponding connected Lie subgroup $H$ of $G$ is semisimple, without compact factor, has finite
center, and its orbit through $o$ is $\Y$. Let $Z_G(H)$ be the centralizer of $H$ in $G$.

\def\qg{\mathfrak q}
\def\ql{\mathfrak l}
\def\qm{\mathfrak m}
\def\ug{\mathfrak u}
\def\sg{\mathfrak s}
\def\tg{\mathfrak t}

\begin{fact}\label{stab} Let $S_0$ be the connected component of the stabilizer $S$ of ${\cal Y}$ in
  $G$. Then  $S_0=HU$ where $U$ is the connected component of $Z_G(H)\cap K$. 
\end{fact}

\begin{proof} 
If $u\in U=(Z_G(H)\cap K)_0$ then certainly $u\in S_0$ since $u_{|\Y}={\rm id}_\Y$. Hence $HU<S_0$.  

The group $\phi(S)$ is stable by the Cartan involution of ${\rm
  Isom}_0({\cal X})$ defined by conjugacy by the geodesic symmetry $s_o$ w.r.t. the point $o\in\cal Y$,
because ${\cal Y}$ being totally geodesic it is preserved by the symmetries w.r.t. its points. 
Therefore the Lie algebra $\mathfrak s$ of $S$ is stable under the corresponding Cartan involution
of $\gg$. 
Hence we have $\sg=\mg\oplus\qg$, where $\mg=\sg\cap\kg$ is a subalgebra of $\kg$ containing 
$\lg$ and 
$\qg=\hg\cap\pg=\sg\cap\pg$ because $\qg$ can be identified with the tangent space at $o$ of the
orbit $\cal Y$ of $o$ under $H$ which is also the orbit of $o$ under $S$.   
The fact that $[\qg,\qg]=\lg$ and $[\mg,\qg]\subset\qg$ implies that $\lg$ is an ideal in $\mg$ (hence
$\hg$ is an ideal in $\sg$), which in
turn implies that the orthogonal $\ug$ of $\lg$ in $\mg$ for the Killing form of $\gg$ is an ideal
in $\mg$ 
hence in $\sg$. Therefore $\sg=\ug\oplus\hg$ is a direct sum of ideals. Now $\u$ is
included in the Lie algebra of $Z_G(H)\cap K$, hence the result.    
\end{proof}

\begin{fact}\label{cent}
 The following assertions are equivalent:

-- the symmetric subspace $\cal Y=H o\subset\X$ is not a factor;

-- the centralizer $Z_G(H)$ of $H$ in $G$ is a subgroup of $K$; 

-- the subgroup $H$ of $G$ has only one totally geodesic orbit in $\cal X$.  
\end{fact}

\begin{proof} 
Suppose first that $\cal Y$ is a factor. This means that there exists a totally geodesic isometric embedding
$f:\cal Y\times\R\fd\cal X$ such that $f(y,0)=y$ for all $y\in\Y$. Hence there exists
$v$ of unit norm in the
orthogonal complement $\qg^\perp$ of $\qg$ in $\pg$ such that $[v,\qg]=0$ (geometrically, and if we
identify $\pg$ with $T_o\cal X$ and $\qg$ with $T_o\cal Y$, 
$-\|[v,x]\|^2$ is the sectional curvature of the 2-plane generated by two orthonormal vectors $x$
and $v$. If $x\in\qg$, 
this is zero because $x$ and $v$ belong to different factors of a Riemannian product). Since
$\hg=[\qg,\qg]\oplus\qg$, $v$ commutes with $\hg$. Hence $H$ commutes with the noncompact
1-parameter subgroup 
of transvections along the geodesic defined by $v$. 

Assume now that the connected component $Z_G(H)$ of the centralizer of $H$ in $G$ is not
included in $K$ and let us prove that $H$ has (at least) two distinct, hence disjoint, totally
geodesic orbits $\cal Y$ and $z\cal Y=zHo=Hzo$ for some $z\in Z_G(H)$. Let indeed $z\in Z_G(H)$ and
suppose by contradiction that $z\Y\cap \Y\neq\emptyset$, i.e. there exists
$y_0\in\Y$ such that   
$zy_0\in\Y$. Then $z$ stabilises $\Y=H y_0$ and, if $d$ is the distance in $\X$, for all $h\in H$, we have
$d(hy_0,zhy_0)=d(hy_0,hzy_0)=d(y_0,zy_0)$, which means that $y\mapsto d(y,zy)$ is constant on
$\Y$, equal to $t_z$ say. If $t_z>0$, $z$ acts on $\Y$
as a non trivial Clifford translation. This implies that $\Y$ splits a line, that is $\Y$ is
isometric to a product $\cal Z\times\R$, see e.g.~\cite[p. 235]{BH}. This is not possible since $H$ is
semisimple. Hence $t_z=0$, 
so that $z$ fixes $\Y$ pointwise and belongs to $K$.

Finally assume that $H$ has 
two distinct totally geodesic orbits $\cal Y=Ho$ and $\cal Y'=Hgo$ for some $g\in G$. Then, 
$d$ being  the distance in $\X$, the
function $x\mapsto d(x,\cal Y')$ is convex on $\X$ because $\cal Y'$ is
totally geodesic. Its restriction to $\cal Y$  is bounded by $d(o,go)$, hence it is constant, equal
to $a$ say, 
because $\cal Y$ is also totally geodesic. Therefore the convex hull of these two orbits is isometric to
$\cal Y\times [0,a]$ and $\cal Y$ is a factor. See e.g. ~\cite[Chap. II.2]{BH}.
\end{proof}

\begin{fact}\label{geod}
Assume that $Z_G(H)$ is compact and that $L$ is a connected Lie subgroup of $G$ containing $H$. Then
$L$ has a totally geodesic orbit in $\X$.
\end{fact}

\begin{proof}
It is enough to show that the Lie algebra $\lg$ of $L$ is stable by a Cartan involution of
$\gg$. By~\cite[Lemma~1.5]{BHC}, since $G$ is a connected linear semisimple Lie group and is
therefore the connected component of the real points ${\mathbf G}(\R)$ of an algebraic group
${\mathbf G} $ defined over $\R$, it
suffices to prove 
that $\lg$ is reductive and algebraic in $\gg$. We are going to show that $\lg=\sg\oplus\ag$, where
$\sg$ and $\ag$ are ideals of $\lg$, $\sg$ is semisimple, $\ag$ is abelian and all the
elements of $\ag$ are semisimple for the adjoint action of $\ag$ on $\gg$. This will imply that
$\lg$ is reductive in $\gg$, i.e.\ that the adjoint action of $\lg$ on $\gg$ is semisimple. Since
moreover in this case $\ag\subset Z_\gg(\lg)\subset Z_\gg(\hg)$ is a compact subalgebra, $\lg$ is 
indeed algebraic.       

\smallskip

The desired decomposition $\lg=\sg\oplus\ag$ will be established if we prove that $L$ does not
normalize any Lie subalgebra of $\gg$ containing only nilpotent elements. 

Indeed, assuming the latter, let $\rg$ be the radical of $\lg$. First, $[\rg,\rg]$ must be trivial 
  since it only contains nilpotent elements and it is normalized by $L$. Therefore $\rg$ is abelian. Now,
  we note that if $r=r_s+r_n$ and $r'=r'_s+r'_n$ are elements of $\rg$ written in terms of their
  Jordan-Chevalley decomposition, then we also have $[r_n,r'_n]=0$. Indeed,
  $[r_s,r']+[r_n,r']=[r,r']=0$. As ${\rm ad}(r_s)$ and ${\rm ad}(r_n)$ are polynomials in ${\rm
    ad}(r)$, and since ${\rm ad}(r_n)$ is nilpotent, we must have $[r_s,r']=[r_n,r']=0$. Reasoning
  in the same way with $[r_n,r']$, we see in particular that $[r_n,r'_n]=0$. 

Then $\rg_n:=\{r_n\,|\,r\in\rg\}$ is a subalgebra of $\gg$ containing only nilpotents elements and it is normalized by
$L$ since for any $g\in G$ and any $r\in\rg$, $\bigl({\rm Ad}(g)(r)\bigr)_n={\rm Ad}(g)(r_n)$. 
As a consequence, $\rg_n$ is trivial, i.e. $\rg$ is abelian and only contains semisimple
elements. Finally, as $\rg$ is an ideal in $\lg$, its adjoint action on $\lg$ is nilpotent hence
trivial, i.e.\ $\rg$ is central in $\lg$.

\smallskip

To conclude, assume for the sake of contradiction that $L$
normalizes a non trivial Lie subalgebra $\ug$ which only contains nilpotent elements, and let $U$ be
the unipotent subgroup of $G$ whose Lie algebra is $\ug$. Then $U$ is the set of real points of a connected
algebraic unipotent subgroup $\bf U$ of $\bf G$ defined over $\R$.
By~\cite[Corollaire 3.9]{BT}, there exists 
a proper parabolic subgroup $\bf P$ of $\mathbf G$ defined over $\R$ whose
unipotent radical we denote by $\bf R$, such that 
$\bf U\subset \bf R$ and $N_{\bf G}(\bf U)\subset \bf P$ where $N_{\bf G}(\bf U)$ is the normalizer
of $\bf U$ in $\bf G$. In
particular, we have $H\subset L\subset N_{G}(U)\subset N_{\bf G}(\bf U)(\R)\subset \bf P(\R)$. 

Now, by~\cite[Corollary 3.14.3]{V} there exists a Levi factor $\mg$ of the Lie algebra of $P:=\bf P(\R)$ which
contains the Lie algebra $\hg$ of $H$. 
Since $P$ is a proper parabolic subgroup of $G$, $Z_G(\mg)$ is noncompact which is impossible since it is
contained in $Z_G(H)$. 
\end{proof}

\begin{fact}\label{finite}
Up to the action of the stabilizer $K$ of the point $o\in\X$, there are only
finitely many symmetric subspaces of the noncompact type in $\X$ passing through $o$. In
particular, up to the action of its isometry group, a symmetric space of the noncompact type  admits
only 
finitely many symmetric subspaces of the noncompact type. As we said in the introduction,
the orbit of a symmetric subspace $\Y$   of $\X$ under $G$ 
  is called the {\em kind} of $\Y$.
\end{fact}

\begin{proof} 
This follows from the fact that there are only finitely many $G$-conjugacy classes of semisimple
subalgebras in the Lie algebra of a connected real Lie group $G$, see
e.g.~\cite[Prop.~12.1]{Richardson}.  Let us
indeed consider a totally geodesic subspace of the noncompact 
type $\Y'$  in the symmetric 
space of the noncompact type $\X$. Up to the action of the isometry group, we may assume that $\Y'$
contains the point $o$. Its tangent space at $o$ then identifies with a Lie 
triple system $\qg'$ of $\pg$ such that $\hg':=[\qg',\qg']\oplus\qg'$ is a semisimple subalgebra of the
Lie algebra $\gg$. Let $H'$ be the corresponding connected subgroup of isometries of $\X$. We have
$\Y'=H'o$. Let $\hg_1,\ldots,\hg_m$ be a system of representatives of the
conjugacy classes of semisimple subalgebras of $\gg$. We may assume that the $\hg_i$ are stable
under the Cartan involution given by the geodesic symmetry around the point $o$ so that if $H_i$ is
the connected subgroup of isometries corresponding to $\hg_i$, the orbit $\Y_i=H_i o$ is a totally
geodesic subspace of $\X$. 

We have $\hg'=g^{-1}\hg_j g$ for some isometry $g$ of $\X$ and some $1\leq j\leq m$, so that
$H'=g^{-1}H_jg$. Then the semisimple subgroup $H_j$ has two totally geodesic orbits $H_jo$ and
$H_j go$ in $\X$. It follows from the proof of Fact~\ref{cent} that there exist a transvection $z$ in the
centralizer of $H_j$, $h\in H_j$ and $k\in K$ such that $g=hzk$. Hence $H'=k^{-1}H_j k$ and
$\Y'=k^{-1}\Y_j$.  
\end{proof}


\begin{fact}\label{div}
  Assume that ${\cal X}$ is a Hermitian symmetric space and that $\cal Y$ is a totally geodesic
  divisor of $\cal X$. Then either $\cal Y$ is not a factor or $\cal X={\cal Y}\times{\mathbb
    H}^1_\C$, where ${\mathbb H}^1_\C$ is the hyperbolic disc. Moreover, if $\Y$ is not a
    factor and $\X=\X_1\times\cdots\times\X_\ell$ is the decomposition of $\X$ in a product of
    irreducible Hermitian symmetric spaces, then either
    \begin{enumerate}
    \item there exist $i\in\{1,\ldots,\ell\}$ and a totally geodesic subspace $\Y_i$ of $\X_i$ such that
      $\Y=\Y_i\times\prod_{j\neq i}\X_j$, or
\item there exists $i\neq j$ in $\{1,\ldots,\ell\}$ and an isometry (up to scaling) $\varphi:
  \X_i\pfd\X_j$ such that 
  $\Y=\{(x,\varphi(x))\mid x\in\X_i\}\times\prod_{k\neq i,j}\X_k$.  
    \end{enumerate}
For dimensional reasons, in the first case $\Y_i$ is a divisor in $\X_i$ and in the second case
$\X_i$ and $\X_j$ are both isometric (up to scaling) to the hyperbolic disc ${\mathbb H}^1_\C$.
  \end{fact}

\begin{proof} 
If $\cal Y$ is a factor then as we saw there exists $v\in\qg^\perp\subset\pg$ which commutes with every
element of $\hg=[\qg,\qg]\oplus\qg$. The complex structure on $\pg$ is given by ${\rm ad}(z)$ for an element $z$ in the
center of $\kg$. Since $\cal Y$ is complex, $\qg$ is invariant by ${\rm ad}(z)$ and hence ${\rm
  ad}(z)v$ also belongs to $\qg^\perp$. Hence $\pg=\qg\oplus\R\, v\oplus\R\, {\rm ad}(z)v$ because
$\dim_\C\qg=\dim_\C\pg-1$. It is easily checked that $\R\, v\oplus\R\, {\rm ad}(z)v$ is a Lie triple
system of $\pg$ and that $\R\, [{\rm ad}(z)v,v]\oplus\R\, v\oplus \R\,{\rm ad}(z)v$ is a Lie
subalgebra of $\gg$ which commutes with $\hg$. This subalgebra is either isomorphic to $\C$ or to
${\mathfrak s\mathfrak l}(2,\R)\simeq{\mathfrak{su}}(1,1)$. Since $\gg$ is semisimple, it is
${\mathfrak s\mathfrak l}(2,\R)$. 

If $\Y$ is not a factor, then it is a {\em maximal} totally geodesic subspace of $\X$, in the
terminology of~\cite{Kollross}, meaning that
  if $\mathcal Z$ is a totally geodesic subspace of $\X$ containing $\Y$ then either $\mathcal Z=\Y$
  or $\mathcal Z=\X$. Indeed, if $\mathcal Z\neq \Y$ then $\mathcal Z$ is a (real) hypersurface of
  $\X$. By~\cite[Corollary~3.5]{Kollross}, totally geodesic hypersurfaces of $\X$ must be of the
  form $\mathcal Z_i\times\prod_{j\neq i} \X_j$ where $\mathcal Z_i$ is a totally geodesic
  hypersurface of $\X_i$. Then necessarily $\X_i$ has constant sectional curvature and since it is a
  Hermitian symmetric space, it must be isometric to the disc ${\mathbb H}^1_\C$. Therefore
  ${\mathcal Z}\simeq \R\times\prod_{j\neq i}\X_j$ and by the same argument either $\Y\simeq\{{\rm
    pt}\}\times\prod_{j\neq i}\X_j$, or there is a second factor $\X_j$, $j\neq i$, isometric to a
  disc ${\mathbb H}^1_\C$ and $\Y\simeq \R^2\times\prod_{k\neq i,j}\X_k$. That's a contradiction
  since in the former case $\Y$ is
  a factor while in the latter it is not a complex submanifold of $\X$. Hence $\Y$ is indeed maximal
  and we may apply~\cite[Theorem~3.4]{Kollross} which exactly gives the alternative in our statement.
\end{proof}

\section{Proofs}

\subsection{The main ingredient}\label{BenoistQuint}\hfill

\medskip

Theorem~\ref{ratner} follows from considerations
originating in the celebrated results of M. Ratner on unipotent flows, see e.g. \cite{R} for a
survey. The key result we are going to use is Y.~Benoist and J.-F. Quint Theorem 1.5 in~\cite{BQ3}.

Let us begin by giving the definitions needed to  
quote a downgraded version of~\cite [Theorem 1.5]{BQ3}. Our notation are a bit different from those
of~\cite{BQ3}. Let $G$ be a real Lie group, $\G$ a lattice in $G$, and $H$ a Lie subgroup of
$G$ such that $\Ad(H)$ is a semisimple subgroup of $\rm{GL}(\gg)$ with no compact factors.    

A closed subset $Z$ of $\G\backslash G$ 
is called a finite volume homogeneous subspace if the stabilizer $G_Z$ of $Z$ in $G$
acts transitively on $Z$ and preserves a Borel probability measure $\mu_Z$ on
$Z$. If moreover $G_Z$ contains $H$, $Z$ is said $H$-ergodic if $H$
acts ergodically on $(Z, \mu_Z)$.

Let $\comp\subset \G\backslash G$ be a compact subset of $\G\backslash G$ and $E_{\comp}(H)$ be
  the set of  
$H$-invariant and $H$-ergodic finite volume
homogeneous subspaces $Z$ of $\G\backslash G$ such that $Z\cap\comp\not=\emptyset$. 
We may identify $E_{\comp}(H)$ with a set of Borel probability measures on
$\G\backslash G$ through the map $Z\mapsto \mu_Z$. In particular $E_{\comp}(H)$ is endowed with
the topology of weak convergence, so that a sequence $(Z_n)$ in $E_{\comp}(H)$
converges toward $Z\in E_{\comp}(H)$ if and only if $\mu_{Z_n}$
converges toward $\mu_{Z}$.

\smallskip

Then~\cite [Theorem 1.5]{BQ3} implies the following:

\begin{theo}\label{BQ}
Let $G$ be a real Lie group, $\G$ a lattice in $G$, and $H$ a Lie subgroup of
$G$ such that $\Ad(H)$ is a semisimple subgroup of $\rm{GL}(\gg)$ with no compact factors. Let $\comp\subset \G\backslash G$ be a compact subset. Then
\begin{enumerate}
\item the space $E_{\comp}(H)$ is compact; 
\item if $(Z_n)$ is a sequence of $E_{\comp}(H)$ converging to $Z\in E_{\comp}(H)$, there exists a 
  sequence $(\ell_n)$ of elements of the centralizer of $H$ in $G$ such that $Z_n\cdot\ell_n\subset Z$
  for $n$ large.
\end{enumerate}
\end{theo}

\subsection{Proof of Theorem~\ref{ratner}} \hfil

\medskip 

We recall the notation from the introduction. We have a symmetric space of the noncompact type $\X$
and $G$ is the connected component of the isometry group of $\X$. In
particular $G$ is a connected semisimple real Lie group with trivial center and without compact
factor. We fix an origin $o\in\X$ and let $K$ be the isotropy group of $G$ at $o$. Moreover $\G$ is 
a torsion-free lattice 
of $G$ and $X$ is the finite volume locally symmetric space $\G\backslash\cal X$.  

Let $(Y_j)$ be a sequence of compact totally geodesic
submanifolds of the noncompact type of $X$ as in the statement of the theorem. 

\bigskip 

The first thing to remark is that by the finiteness result of Fact~\ref{finite},
 up to extracting subsequences, we may 
assume that the submanifolds $Y_j$ are all of the same kind, meaning that they are of the
form $\G\backslash\G g_j\Y$, where $\Y$ is a fixed 
symmetric subspace of $\X$ passing through the origin $o$, and $g_j\in G$ are such that   
$\G \cap g_j S g_j^{-1}$ is a lattice in $g_jS g_j^{-1}$,  $S$ being the
stabilizer of $\Y$ in $G$. We call $H$ the connected semisimple subgroup without compact factor and
with finite center such that $\Y=H o$.  

Moreover, by Fact~\ref{cent}, the hypothesis that the
totally geodesic submanifolds $Y_j$ are not local factors implies that
the centralizer $Z_G(H)$ of $H$ in $G$ is included in $K$  and that $\cal Y$ is the only totally
geodesic orbit of $H$ in $\X$. 

\bigskip

Consider the (right) $S$-invariant subsets $\Gamma\backslash\Gamma g_j S$ in $\Gamma\backslash
G$. They 
support natural $S$-invariant probability measures, but these measures might be non ergodic with
respect to the action of $H\subset S$. To get rid of this problem, we need to consider the action of
$H$ on the orbit of a smaller
subgroup than $S$. Let $S_0$ be the connected component of the stabilizer $S$ of $\Y$. By Fact~\ref{stab},
there exists a subgroup $U<K$ centralizing $H$ such that 
$S_0=HU$. The intersection $H\cap U$ is the center of $H$ and hence is finite. The intersection
$g_j^{-1}\G g_j\cap S$ is by assumption a lattice in $S$ and therefore the intersection $g_j^{-1}\G
g_j\cap S_0$ is a lattice in $S_0$, since $S_0$ has finite index in $S$.  Let $M_j$ be the 
``projection'' of $g_j^{-1}\G g_j\cap S_0$ to $U$, namely the group $\{u\in U\mbox{ such that } u=\g
h\mbox{ for some }
\g\in g_j^{-1}\G g_j\cap S\mbox{ and } h\in H\}$, and let $\bar M_j$ be the
closure of $M_j$. Then $\bar M_j$ is a compact subgroup
of $U$ and we let $S_j:=\bar M_j H$. This time, the right action of $H$ on the $S_j$-invariant
probability measure $\mu_j$ supported on $Z_j:=\Gamma \backslash\Gamma g_j S_j$ is
ergodic. Indeed, 
a $H$-orbit in $Z_j$ is the same as a (left) $M_j H$-orbit in
$S_j$ (because $H$ is obviously normal in $S_j$) and
the group 
$M_j H$ is dense in $S_j$ by construction
(see \cite[Prop. I.(4.5.1)]{Ma}). 
Notice however that the push forward of the measure $\mu_j$ by the projection $\pi:\G\backslash G\fd
X=\G\backslash G/K$ is the same as the push forward of the $S$-invariant probability measure on
$\G\backslash\G g_j S$ by $\pi$. Indeed, $g_jSg_j^{-1}$ (resp. $g_jS_jg_j^{-1}$) is unimodular because it contains its intersection with $\Gamma$ as a lattice, and its Haar measure suitably normalized induces the $S$-invariant (resp. $S_j$-invariant) probability measure on $\Gamma \backslash\G g_j S$ (resp. $\Gamma \backslash\G g_j S_j$). The push forward of these Haar measures define $g_jHg_j^{-1}$-invariant measures supported on $g_j\Y$ and hence they must be proportional. Finally, as they both induce probability measures on $Y_j$, they are equal.
\medskip

We saw in
the proof of Fact~\ref{geod} that $H$ is not contained in any proper parabolic subgroup of $G$ since
$Z_G(H)$ is compact, so that a fortiori for any $g\in G$, $gHg^{-1}$ is not contained in any proper $\G$-rational
parabolic subgroup of $G$. Hence we may apply~\cite[Theorem 2.1]{EM} with a 
probability measure supported on a bounded open subset
of $H$. This result implies in particular that there is a compact subset $\comp$ of $\G\backslash
G$ such that for all $\G g_j\in{\rm
  supp}\,\mu_j$, there exists $h_1,\ldots,h_m\in H$ such that $\G g_j h_1\dots
h_m\in\comp$. Therefore the measures $\mu_j$ belong to the set $E_{\comp}(H)$ of 
$H$-invariant and  
$H$-ergodic finite volume homogeneous subspaces 
of $\G\backslash G$ intersecting $\comp$, which is compact by Theorem~\ref{BQ}.

Hence, after extraction of a
subsequence, the sequence $(\mu_j)$ 
converges weakly to a $H$-ergodic probability measure $\mu$ whose support ${\rm supp}\,\mu$
is a closed $G^\mu$-homogeneous subset of $\G\backslash G$,  where $G^\mu:=\{g\in G : \mu g=\mu\}$
is a (closed) Lie subgroup of 
$G$ containing $H$. Moreover, by the second part of the theorem, there exists a sequence $(\ell_j)$ of
elements of  
the centralizer of $H$ in $G$ such that for $j$ large enough, ${\rm supp}\,\mu_j\subset
({\rm supp}\,\mu)\cdot \ell_j$. 

\medskip

We are going to prove that $G^\mu=G$ and the compactness of the centralizer of $H$ will come into
play in order to neutralize the effect of the $\ell_j$'s.

By fact~\ref{geod}, the connected component $G_0^\mu$ of $G^\mu$ has a totally 
geodesic orbit $G_0^\mu x^\mu$ in $\cal X=G/K$ for some $x^\mu\in \X$. Moreover, we know that ${\rm
  supp}\,\mu=\G\backslash\Gamma g G^\mu$ 
for some $g\in G$ and we have seen above that if $j$ is large, then $Z_j\ell_j=\Gamma \backslash
\Gamma g_j \bar M_j H \ell_j 
\subset \G\backslash\Gamma g  G^\mu$ for some $\ell_j$ in the centralizer of $H$. Hence
$\Gamma\backslash \Gamma g_jHK/K \subset \G\backslash\Gamma g G_0^\mu K/ K$, 
because $\ell_j\in K$ by 
assumption. Therefore, there exists 
$\g_j\in\G$ such that 
$\g_j g_j \Y\subset g G_0^\mu o\subset \cal X$.

 This actually implies that, for $j$ large, the totally geodesic submanifolds $\g_j g_j \Y$
are all included in the totally geodesic submanifold $g G_0^\mu x^\mu$ and thus that $G_0^\mu
o=G_0^\mu x^\mu$.

Let indeed $d$ be the distance in the symmetric space $\cal X$. The function 
$$
\begin{array}{rcl}
\cal X & \fd & \R \\
x & \longmapsto & d(x,gG^\mu_0 x^\mu)
\end{array}
$$
is convex because $gG^\mu_0 x^\mu$ is totally geodesic. Its restriction to $gG^\mu_0o$ is bounded
because for all $a\in G^\mu_0$, $d(gao,gax^\mu)=d(o,x^\mu)$. Therefore its restriction to $\g_j g_j\Y$,
which is also totally geodesic, is both convex and bounded, hence constant equal  
to some $d_j\in\R$. Then, if $d_j\neq 0$, there
exists an isometric embedding from the product $(\g_j g_j\Y)\times\R$ to $\cal X$
which maps $(\g_j g_j\Y)\times\{0\}$ to $\g_j g_j\Y$ and $(\g_j g_j\Y)\times \{d_j\}$ to a totally
geodesic subspace of $gG^\mu_0 x^\mu$, see 
e.g.~\cite[Chap. II.2]{BH}. This is a contradiction with the fact that $\cal Y$ is not a factor,
hence $d_j=0$ for all $j$ large, as claimed.    

As a consequence, $\G\backslash\Gamma g  G_0^\mu  o$
is a closed totally geodesic submanifold in $X$ which contains the submanifolds $Y_j$ (for $j$ large
enough). Since no subsequence of $(Y_j)$ is contained in a closed totally geodesic proper submanifold of
$X$,  $\G \backslash\Gamma g G_0^\mu  o=\G\backslash G/K$ which implies
$G_0^\mu=G$ since $G_0^\mu$ is reductive.  

\medskip

In conclusion, $(\mu_j)$ is a sequence of elements of $E_{\comp}(H)$ which
is compact, and its sole limit point is the unique $G$-invariant probability measure on
$\G\backslash G$, so that it converges to this unique measure. Hence the sequence $(\mu_{Y_j})$
converges to $\mu_X$.

\begin{rema}\label{convsub}
It is straightforward that for any compact subgroup $L\subset K$, the pushforward of the measures
$\mu_j$ to $\G\backslash G/L$ converges towards the pushforward to $\G\backslash G/L$ of
the $G$-invariant probability measure on $\G\backslash G$. 
\end{rema}

\subsection{Proof of Corollary~\ref{ratnercurrents}}\label{proofRC}\hfill
\medskip

The kind of the $Y_j$'s is fixed, so that they are of the form $\G\backslash\G g_j\Y$, where $\Y$ is a fixed Hermitian
symmetric subspace of $\X$ passing through the origin $o$, and $g_j\in G$ are such that   
$\G \cap g_j S g_j^{-1}$ is a lattice in $g_jS g_j^{-1}$,  $S$ being the
stabilizer of $\Y$ in $G$. 

We call $\stabY\subset K$ the compact subgroup $\stabY=K\cap S$. The
homogeneous space $G/K_\Y$ is the $G$-orbit of $T_o\Y$ in the Grassmann manifold of complex
$p$-planes in the tangent 
bundle $T\X$ of $\X$. It is a bundle over $\X$ and we let $\hat X:=\Gamma\backslash G/\stabY$ be the
corresponding Grassmann bundle over $X$.

Every totally geodesic manifold $Y_j$ has a natural lift $\hat Y_j$
to $\hat X$:  a smooth point $y$ of $Y_j$ defines the point $\hat y=T_y Y_j$ in $\hat X$. In fact, $\hat Y_j$ is smooth and isometric to $(\G\cap g_j S g_j^{-1})\backslash g_j\Y$ and the
  natural morphism $\nu_j:\hat Y_j\fd Y_j$ is an immersion which is generically one-to-one. 

For each
  $j$, we denote by $\mu_{\hat Y_j}$ the probability measure on $\hat X$ which 
is obtained by taking the direct image of the measure $\mu_j$ on $\G\backslash G$ defined in the
proof of Theorem~\ref{ratner}. We emphasize that the support of the measure $\mu_{\hat Y_j}$ is
indeed $\hat Y_j$. Let $\om_X$ be the K\"ahler form on $X$ induced by a
$G$-invariant K\"ahler form $\om$ on $\X$. Define ${\rm vol}(Y_j)=\frac{1}{p!}\int_{ Y_j}\om_X^p$ where
$\int_{Y_j}$ means integration over the smooth part of $Y_j$, i.e. ${\rm
  vol}(Y_j)=\frac{1}{p!}\int_{ \hat Y_j}\nu_j^*\om_X^p$.  
Then the probability measure with support $\hat Y_j$ and
density $\frac  
1{p!\,{\rm vol}(Y_j)}\,\nu_j^*\om_X^p$ is equal to $\mu_{\hat Y_j}$. This is
again due to the fact that they are both induced by $g_j H g_j^{-1}$-invariant measures on the orbit
$g_j HK_\Y\subset G/K_\Y$.

For any $(p,p)$-form $\eta$ on $\X$, we define a function $\varphi_\eta$ on $G$ by
$$\varphi_\eta(g):=p!\frac{\eta(e_1,\dots,e_{2p})}{\om^p(e_1,\dots,e_{2p})}$$ 
where $(e_1,\dots,e_{2p})$ is any basis of $T_{gK}g\Y$. Said another way, in restriction to $T_{gK}g\Y$ 
the two $(p,p)$-forms $\eta$ and $\frac1{p!}\om^p$ are proportional, and
$\varphi_\eta(g)$ is the coefficient of proportionality. 
Note that the action of $\stabY$ on $G$ by right multiplication induces the trivial action on
$\varphi_\eta$ by construction so that we will see it as a function on $G/\stabY$ as well. Moreover,
if $\eta $ is the lift of a form on $X$ then $\varphi_\eta$ is well defined on $\G\backslash G$ (and
$\hat X=\G\backslash G/\stabY$). Then by the proof of Theorem~\ref{ratner} and
Remark~\ref{convsub} we have  
$$\frac 1{{\rm vol}(Y_j)}\int_{Y_j} \eta=\frac 1{{\rm vol}(Y_j)}\int_{\hat Y_j} \nu_j^*\eta=\frac
1{p!{\rm vol}(Y_j)}\int_{\hat Y_j} \varphi_\eta\,\nu_j^*\om_X^p=\int_{\hat
  X}\varphi_\eta\,d\mu_{\hat Y_j}\mathop{\longrightarrow}\limits_{j\rightarrow +\infty} \int_{\hat
  X}\varphi_\eta\,d\mu_{\hat X} 
$$
where $d\mu_{\hat X} $ is the probability measure on $\hat X$ induced by the Haar measure $dg$ on $G$
normalized in such a way that $\int_{\G\backslash G}dg=1$.

For any $(p,p)$-form $\eta$, we also define a function $\psi_\eta$ on $G$ by
$$\psi_\eta(g):=\int_K\varphi_\eta(gk) dk 
$$
where $dk$ is the Haar probability measure on $K$. Actually, $\psi_\eta$ is well defined on $\X=G/K$
and in the same way as $\varphi_\eta$, it can be seen as a function on $X=\G\backslash G/K$ if
$\eta$ comes from $X$. 

Now, as
$$\int_{\Gamma\backslash G} \varphi_\eta(gk) dg=\int_{\G\backslash G}\varphi_\eta(g)dg
$$
for any $k\in K$ (just because $dg$ is right invariant), we get for any $(p,p)$-form $\eta$ on $X$
$$\int_{\hat X}\varphi_\eta \,d\mu_{\hat X}=\int_{\G\backslash G}\varphi_\eta(g)\, dg =
\int_K\int_{\G\backslash G}\varphi_\eta(gk)\, dg \,dk=\int_{\G\backslash G}\psi_\eta(g)\,
dg=\frac1{n!{\rm vol}(X)}\int_{X}\psi_\eta \,\om_X^n. 
$$

Let us consider the linear form $\eta\mapsto \psi_\eta(e)$ on the space of $(p,p)$-forms on $\X$. It
only depends on $\eta(o)$, hence there exists a $(p,p)$-form $\Psi$ on $T_o{\X}$ such that
$\psi_\eta(e)=\la \eta,\Psi\ra_o$. Moreover, for any $k\in K$, $\psi_{\eta}(k)=\psi_{k^*\eta}(e)=\la
k^*\eta,\Psi\ra_o=\la \eta,(k^{-1})^*\Psi\ra_o$. As $\psi_\eta$ is $K$-invariant, we conclude that
$\Psi$ is also $K$-invariant. Therefore, $\Psi$ is the restriction of a (unique) $G$-invariant form
on $\X$ that we still denote by $\Psi$. Then, for any $g\in G$,
$\psi_{\eta}(g)=\psi_{g^*\eta}(e)=\la g^*\eta,\Psi\ra_o=\la
\eta,(g^{-1})^*\Psi\ra_{go}=\la\eta,\Psi\ra_{go}$ i.e. $\psi_\eta=\la \eta,\Psi\ra$ on $\X$ for any
$\eta$. 
Finally,
$$
\frac1{n!{\rm vol}(X)}\int_{X}\psi_\eta\, \om_X^n=\frac1{n!{\rm vol}(X)}\int_{X}\la
\eta,\Psi\ra\,\om_X^n=\frac 1{{\rm vol}(X)}\int_X \eta\wedge \star \Psi 
$$
where $\star$ is the Hodge star operator. Setting $\mm:=\star \Psi$, we get the desired result. The form
$\mm$ is closed and positive since this is the case for each current of integration over $Y_j$
(notice that actually, a $G$-invariant form of even degree is automatically closed as its
differential is a $G$-invariant form of odd degree and $G$ contains an involution admitting $o$ as
an isolated fixed point i.e. its differential at $o$ is $-{\rm id}_{T_o \X}$). 

\medskip

By construction, $\mm$ only depends on $\Y$, i.e. on the kind of the $Y_j$'s, and on the
  choice of a $G$-invariant Kähler form $\om$ 
  on $\X$. If $\X$ is irreducible, then there is only one such choice up to a positive constant.    
If $\X$ is not irreducible, then this is not the case. However, if $\om'$ is another $G$-invariant
K\"ahler form, the restriction of $\om^p$ and 
${\om'}^{p}$ to $\Y$ only differ by a multiplicative positive constant $c$. Indeed,
this is clear at $o$ and the two forms are $S$-invariant. As a consequence, the corresponding
functions $\varphi_\eta$ and $\psi_\eta$ differ by the constant $1/c$ and so the resulting forms $\mm$ only
differ by a positive constant.

\medskip

In general, if $(\eta_1,\dots,\eta_m)$ is an orthonormal basis of $G$-invariant $(p,p)$-forms
on $\X$, then it is straightforward from the construction above that $\mm=\sum_{i=1}^m
\varphi_{\eta_i}\,(\star\eta_i)$, the $\varphi_{\eta_i}$ being constant since $\eta_i$ is $G$-invariant.

Assume now that the $Y_j$'s are divisors, that is $p=n-1$. 

Let $\X=\X_1\times\dots\times\X_\ell$ be the decomposition of $\X$ in a product of irreducible
Hermitian symmetric spaces $\X_i$ of dimension $n_i$ associated to isometry groups $G_i$, and
$\om_i$ the unique (up to a positive constant) $G_i$-invariant K\"ahler form on $\X_i$. If the
$Y_j$'s are divisors, then $\mm$ is induced by $\sum_{i=1}^\ell a_i\,\om_i$ for some non-negative real
numbers $a_i$. The positive $(n-1,n-1)$-forms $\eta_i:=\om_i^{n_i-1}\wedge\bigwedge_{j\not=
  i}\om_j^{n_j}$ make up an orthogonal basis of $G$-invariant $(n-1,n-1)$-forms on $X$.

By fact~\ref{div}, one can assume that $\Y=\D\times\X_{k+1}\times\dots\times\X_\ell$, where either
$k=1$ and $\D$ is a divisor in $\X_1$, or $k=2$, $\X_1=\X_2=\Ho$ and $\D\simeq\Ho$ which is
diagonally embedded in $\X_1\times\X_2$.

In the first case, the restriction of all the $\eta_i$'s to $\Y$ vanishes, except when $i=1$ and
this implies that only $a_1>0$. Similarly, in the second case, only $\eta_1$ and $\eta_2$ do not
vanish on $\Y$ and hence only $a_1$ and $a_2$ are positive (and equal if $\om_1$ and $\om_2$ are
chosen in such a way that $\X_1$ and $\X_2$ are isometric). 

As a consequence, $\mm$ is positive (as a (1,1)-form) only if $\X$ is irreducible or if
$\X=\Ho\times\Ho$ and $\Y\simeq\Ho$ is diagonally embedded. However, in all cases, the restriction of
$\mm$ to $Y_j$ never vanishes.

\medskip

\begin{rema}
In the case of the $n$-ball, it is well known that for any $1\leq p\leq n$, the
space of $G$-invariant $(p,p)$-forms is 1-dimensional and generated by $\om^p$ hence in this case we
always have $\mm=\frac{p!}{n!}\,\om_X^p$.  
\end{rema}

\subsection{Proof of Theorem~\ref{negativecurves}}\label{ProofNC}\hfill

\medskip

We only prove the second assertion of the theorem since the first is just a particular case by
Grauert's criterion, see for example~\cite[p. 91]{BHPV}. 

\smallskip

Recall that a divisor $D$ in a compact complex manifold $X$ is {\it exceptional} if there exists a
neighborhood $U$ of $D$ in $X$, a proper bimeromorphic map $\phi:U\fd U'$ onto a (possibly singular)
analytic space $U'$ and a point $x'\in U'$ such 
that $\phi(D)=\{x'\}$, and $\phi$ induces a biholomorphism between $U\backslash D$ and
$U'\backslash\{x'\}$.

\medskip

Let $\X$ be the universal cover of the Hermitian locally symmetric space $X$, $G$ the
connected component of the isometry group of $\X$ and $\G$ the torsion-free lattice of $G$
such that $X=\G\backslash\X$.  

Assume that there exist infinitely many
totally geodesic (irreducible) exceptional divisors $(D_j)_{j\in\N}$ in $X$. 
As before, because of Fact~\ref{finite}, we may assume that the totally geodesic divisors $D_j$ are
of the form $\G\backslash\G g_j\cal D$, where $\cal D$ is a totally geodesic divisor of $\X$
containing a fixed point $o\in\X$. 

 By Fact~\ref{div}, either $D_j$ is not a local factor or
$\X\simeq g_j \cal D\times\Ho$. In the latter case, $G\simeq H_j\times \PU(1,1)$, where $H_j$ is
the connected component of the isometry group of $g_j\cal D$, and this implies that the lattice $\G$
is not irreducible. Indeed, if $S_j$ is the stabilizer of $g_j\cal D$ then by assumption $\Gamma\cap
S_j$ is a lattice in $S_j=H_j\times U$ where $U$ is a compact subgroup of 
$\PU(1,1)$. By~\cite[Thm 1.13]{Rag}, $\G\cdot (H_j\times U)$ is closed in $G$. This is not possible
if $\G$ is irreducible because then the projection of $\G$ onto $\PU(1,1)$ is dense, so that
$\G\cdot(H_j\times U)$ is dense in $G=H_j\times\PU(1,1)$. Therefore, up to a finite covering, we have
$X=D_j\times \Sigma_j$ (where $\Sigma_j$ is a curve) and in 
particular, $D_j$ is not exceptional. 

\medskip

Hence we may assume that the $D_j$'s are of the same kind and are not local factors, so that
  Corollary~\ref{ratnercurrents} applies.   

Let $\om_X$ be a Kähler form on $X$ induced by a $G$-invariant Kähler form of $\X$. We may choose
$\om_X$ so that it represents the first Chern class $c_1(K_X)$ of the canonical bundle $K_X$. All 
volumes will be computed w.r.t. $\om_X$.

Each of the divisors $D_j$ defines an integral class $[D_j]\in H^2(X,\Z)\cap H^{1,1}(X,\R)$. We
write $[D_j]\cdot [\eta]=\int_{D_j}\eta$ for any class $[\eta]\in H^{n-1,n-1}(X,\R)$ and we set
$A_j:=\frac{{\rm vol}(D_j)}{n\,{\rm vol}(X)}=\frac{[D_j]\cdot[\omega_X]^{n-1}}{n!\, {\rm
    vol}(X)}$. Since the pairing between $H^{1,1}(X,\R)$ and 
$H^{n-1,n-1}(X,\R)$ is non degenerate, Corollary~\ref{ratnercurrents} implies that  
$$
\frac 1{A_j} [D_j]\mathop{\longrightarrow}\limits_{j\rightarrow +\infty}[\mm]\leqno{(\star)}
$$
for some closed non-negative $(1,1)$-form $\mm$ which does not vanish in restriction to the $D_j$'s.
The canonical class $[K_X]$ of $X$ is ample and is equal to $[\om_X]$, so let
$m\in\N^\star$ be large enough and $H_1,\dots,H_{n-2}\in |mK_X|$ be irreducible divisors in the
linear system $|mK_X|$ such that $N:=H_1\cap\dots\cap H_{n-2}\subset X$ is a smooth surface
intersecting all the divisors $D_j$ transversally.

Consider now the sequence of curves $(C_j)_{j\in\N}$ of $N$ defined by $C_j=N\cap D_j$. By
definition, for any $j$, there exists a bimeromorphic map $\phi_j:X\fd X_j'$ which contracts the
divisor $D_j$, so that the curve $C_j$ is contracted by the morphism ${\phi_j}_{|N}:N\fd\phi_j(N)$,
hence has negative self intersection by Grauert's criterion (see~\cite[p. 91]{BHPV} for instance).

Let $\Omega_N$ be the restriction of $\mm$ to $N$ and consider now the intersection numbers on $N$
$$
I_j:=\Bigl([\Omega_N]-\frac 1{A_j}[C_j]\Bigr)\cdot \frac 1{A_j}[C_j]=\Bigl([\mm]-\frac
1{A_j}[D_j]\Bigr)\cdot \frac 1{A_j}[D_j]\cdot m^{n-2} [\om_X]^{n-2} 
$$
(here we used $[N]=m^{n-2} [\om_X]^{n-2}$). 

On the one hand, 
$I_j\mathop{\longrightarrow}\limits_{j\rightarrow +\infty} 0$ by $(\star)$, 
and on the other hand, since $C_j^2<0$, we have $I_j\geq[\mm]\cdot\frac 1{A_j}[C_j]=m^{n-2} \frac
1{A_j}[D_j]\cdot [\mm]\cdot [\om_X]^{n-2}=cm^{n-2}{n!\,{\rm vol}(X)}$ for any $j$ and for some
positive constant $c$, a contradiction. We used the fact that since $\mm$ does not vanish in
restriction to the $D_j$'s, the restriction of $\mm\wedge \om_X^{n-2}$ to $D_j$ is equal to $c$
times the restriction of $\om_X^{n-1}$ to $D_j$ for some positive constant $c$, because both are
invariant forms of bidegree $(n-1,n-1)$.

\end{document}